\documentclass[11pt,reqno]{amsart}
\usepackage{amssymb}
\usepackage{txfonts}
\usepackage{amsmath}
\usepackage{multirow}
\usepackage{geometry}
\usepackage{enumerate}

 \newtheorem{thm}{Theorem}[section]

 \theoremstyle{definition}
 \newtheorem{defn}[thm]{Definition}
 \theoremstyle{remark}
 \newtheorem{rem}[thm]{Remark}
  
 \numberwithin{equation}{section}

\emergencystretch=\maxdimen
 \newcommand\ric{Ric}
 \newcommand\cur{Rm}

\begin{document}

\title[]
{A no expanding breather theorem for noncompact Ricci flows}

\author{Liang Cheng and Yongjia Zhang}

\dedicatory{}
\date{}



\thanks{Liang Cheng's Research partially supported by Natural Science Foundation of Hubei 2019CFB511
}

\address{School of Mathematics and Statistics $\&$ Hubei Key Laboratory of Mathematical Sciences, Central  China Normal University, Wuhan, 430079, P.R.China
}

\email{chengliang@mail.ccnu.edu.cn}

\address{ School of Mathematics, University of Minnesota, Twin Cities, MN, 55414, USA}
\email{ zhan7298@umn.edu}

\maketitle

\begin{abstract}
	In this note we show that, under certain curvature positivity conditions (the weak $\operatorname{PIC}-2$ condition or the nonnegative bisectional curvature condition), a complete and noncompact expanding breather of the Ricci flow is also an expanding gradient Ricci soliton. This is the first no expanding breather theorem in the noncompact category.
\end{abstract}

\section{Introduction}

Perelman \cite{P1} proved the no shrinking, steady, and expanding breather theorems on compact manifolds, by applying the monotonicity formulas of his $\mathcal{W}$-functional, $\mathcal{F}$-functional, and normalized $\mathcal{F}$-functional, respectively: these types of breathers must also be gradient Ricci solitons of the corresponding types, which evolve by self-diffeomorphism and scaling. In fact, in the compact category, the steady and expanding gradient Ricci solitons must also be Einstein manifolds. Let us first of all recall the definitions of breathers and solitons.
\begin{defn}
Let $(M,g(t))$ be a complete Ricci flow. If there exist two time instances $t_1<t_2$, a constant $\alpha>0$, and a self-diffeomorphism $\phi: M\rightarrow M$, such that
\begin{eqnarray*}
g(t_1)=\alpha \phi^*g(t_2),
\end{eqnarray*}
then $(M,g(t))$ is called a breather. If $\alpha=1$, $\alpha>1$, or $\alpha<1$, then the breather is called steady, shrinking, or expanding, respectively.
\end{defn}
\begin{defn}
	A Ricci soliton is a tuple $(M, g, X, \lambda)$, where $(M, g)$ is a complete Riemannian manifold, $X$ is a smooth vector field on $M$, and $\lambda>0$ is a constant, satisfying
	\begin{eqnarray*} Rc +\frac{1}{2}\mathcal{L}_Xg =\frac{\lambda}{2}g.
	\end{eqnarray*}
	If $\lambda=0$, $\lambda>0$, or $\lambda<0$, then the soliton is called steady, shrinking, or expanding, respectively. The soliton is called \emph{complete} if the vector field $X$ is complete. If there exists a smooth function $f$ on $M$ such that $X=\nabla f$, then $(M, g, f, \lambda)$ is called a gradient Ricci soliton.
\end{defn}

A complete Ricci soliton $(M,g,X,\lambda)$ always generates a canonical form, that is, a Ricci flow $g(t)=\tau(t)\phi_t^*g$ which moves by self-diffeomorphism and scaling, where $ \tau(t)=\lambda t+1$ and  $\frac{\partial}{\partial t}\phi_t(x)=\frac{1}{ \tau(t)}X(\phi_t(x))$. As indicated by Perelman \cite{P1}, if one views a Ricci flow as an orbit in the space $\operatorname{Met}(M)/ \operatorname{Diff}(M)$, where $\operatorname{Met}(M)$ is the space of all smooth Riemannian metrics and $ \operatorname{Diff}(M)$ stands for the group of self-diffeomorphisms and scalings, then a breather is a periodic orbit and a soliton is a static one. Therefore, the no breather theorem is tantamount to saying that the periodic orbits must also be static.	

Perelman's proofs of the no breather theorems require the existence of minimizers for the $\mathcal{W}$-functional, the $\mathcal{F}$-functional, and the normalized $\mathcal{F}$-functional, respectively. When the manifold is compact, such existence results are straightforward applications of variational problems. A natural question to ask is, under what condition can the no breather theorems be established for noncompact manifolds, and how to carry our the proof. One may certainly attempt to find the minimizers of these functionals, and this is possible under certain geometric conditions. For results of this type, the reader may refer to \cite{RV} and \cite{Zhang1}. We remark here that, in the noncomapct category, this approach is almost impossible for the no steady or expanding breather theorems. The reason is that the $\mathcal{F}$-functional or the normalized $\mathcal{F}$-functional are generally not finite on noncompact steady or expanding gradient Ricci solitons, respectively, when these functionals are evaluated using the potential functions of the corresponding solitons. (One may think of a Bryant soliton for example.)

Another approach was initiated by the result of Lu-Zheng \cite{LZ}, where they constructed a Type I ancient solution using a shrinking breather, and proved, under certain geometric conditions, that the backward blow-down limit of this ancient solution must be the breather itself, which must also be a shrinking gradient Ricci soliton by Naber \cite{N}. This method was refined, and the conclusion improved, by the second author \cite{Zh}, where the condition for the no shrinking breather theorem is reduced to bounded curvature alone. The authors \cite{CZhang} recently further reduced the condition for the no shrinking breather theorem to a lower bound of the Ricci curvature alone.

In this paper, we continue our study in \cite{CZhang} and apply our method to noncompact expanding breathers. Recall that Feldman-Ilmanen-Ni \cite{FIN} established some forward monotonicity formulas for the Ricci flow as the dual version of Perelman's \cite{P1}, whose equalities are fulfilled on expanding gradient Ricci solitons. We will be implementing Feldman-Ilmanen-Ni's forward reduced geometry in this paper in the same way as we have applied  Perelman's reduced geometry in \cite{CZhang}. However, since the forward reduced geometry does not behave as nicely as Perelman's reduced geometry (The reason, intuitively speaking, is this, that on steady or expanding solitons, as it is in the case of a shrinking soliton, the forward reduced volumes should coincide with the $\mathcal{F}$-functional or the normalized $\mathcal{F}$-functional, respectively, evaluated using the corresponding potential functions, whereas the latter two are generally infinite in the noncompact case), we will need to impose some strong curvature conditions.

\begin{thm}\label{main}
Let $(M,g(t))$ be a complete and noncompact expanding breather of the Ricci flow. Assume $g(t)$ has bounded curvature on each time-slice and either one of the following is true.
\begin{enumerate}[(1)]
\item $g(t)$ satisfies the weak $\operatorname{PIC}-2$ condition.
\item $(M,g(t))$ is K\"ahler with nonnegative bisectional curvature.	
\end{enumerate}
Then $(M,g(t))$ is the canonical form of an expanding gradient Ricci soliton.
\end{thm}
\begin{rem}
	Our proof depends heavily on Hamilton's Harnack estimate (\cite{brendle}, \cite{cao}, and \cite{RH2}), and this is why we assumed the above curvature conditions. The reader may easily verify that, if Hamilton's trace Harnack is assumed to be valid, then a nonnegative Ricci curvature assumption is sufficient for the proof. Lott \cite{L} gave an example of a complete but nongradient expanding soliton on noncompact manifold (see page 635 in therein). This soliton has bounded curvature but does not have nonnegative Ricci curvature. Therefore, Theorem \ref{main} cannot be proved in general without any curvature positivity condition.
\end{rem}

Before we conclude our introduction, a word is to be said about our method. Similar to \cite{CZhang}, we constructed a Type III immortal solution starting from the given expanding breather and considered the monotonicity of the quantity
\begin{eqnarray*}
\tilde{\theta}_+(t)=\frac{1}{(4\pi t)^{\frac{n}{2}}}\int_M e^{-\ell_+}d\mu_t,
\end{eqnarray*}
where $\ell_+$ is the forward reduced distance constructed in \cite{FIN}. Note that the monotonicity of this quantity relies on Hamilton's Harnack, as proved in Theorem 4.3 of \cite{LNi}. However, the forward reduced volume defined in \cite{FIN}
\begin{eqnarray*}
\theta_+(t)=\frac{1}{(4\pi t)^{\frac{n}{2}}}\int_M e^{\ell_+}d\mu_t,
\end{eqnarray*}
being automatically decreasing along all Ricci flows on compact manifolds, is yet generally infinite on noncompact manifolds.

The organization of the paper is as follows. In section 2, we review basic  $\mathcal{L}_+$-geometry introduced by Feldman-Ilmanen-Ni \cite{FIN}. In section 3, we estimate  the $\ell_+$-distance on Type III immortal Ricci flows. In section 4, we prove the main theorem. 

\section{Preliminaries}

Let $g(t)$ be a metric evolving by the Ricci flow equation on $M\times[0,T]$. We always assume that either $M$ is compact or $g(t)$ has bounded curvature at each time. Fixing a point $x$, Feldman-Ilmanen-Ni \cite{FIN} defined the following dual version of Perelman's reduced distance, called the \emph{forward reduced distance function}.
\begin{align}\label{l_+length}
\ell_+(y,t)=\frac{1}{2\sqrt{t}}\inf\limits_{ \gamma}\int^t_0\sqrt{\eta}\left(R(\gamma(\eta),\eta)+|\gamma'(\eta)|_{g(\eta)}^2\right)d\eta,
\end{align}
where $(y,t)\in M\times(0,T]$, and the infimum is taken among all piecewise smooth curves $\gamma:[0,t]\rightarrow M$ satisfying $\gamma(0)=x$ and $\gamma(t)=y$. $(x,0)$ is called the \emph{base point} of $\ell_+$. We remark that since the minimizing curve in (\ref{l_+length}) also satisfies an $\mathcal{L}_+$-geodesic equation, whose form is very similar to that of an $\mathcal{L}$-geodesic equation, one may easily modify the arguments in, say, Chapter 7 in \cite{CCGGI}, to verify that $\ell_+$ is locally Lipschitz under our assumption. We then summarize some equations and inequalities satisfied by $\ell_+$. The reader may note their similarity to the case of Perelman's $\mathcal{L}$-geometry.

\begin{thm}[Corollary 2.1 in \cite{FIN}]\label{variation2}
	The $\ell_+$ function satisfies the following equalities for almost every $(y,t)\in M\times(0,T]$ 	\begin{align}
	\frac{\partial \ell_+}{\partial t}=R-\frac{\ell_+}{t}-\frac{K}{2t^{
			\frac{3}{2}}},\label{eq_l_1}\\
	|\nabla \ell_+|^2=\frac{\ell_+}{t}-R+\frac{K}{t^{
			\frac{3}{2}}},\label{eq_l_2}
	\end{align}
	Moreover,  $\ell_+$ satisfies the following inequalities in the barrier sense or in the sense of distribution.
	\begin{align}
	\Delta \ell_+ \leq R+\frac{n}{2t}-\frac{K}{2t^{
			\frac{3}{2}}},\label{eq_l_3}\\
	\frac{\partial \ell_+}{\partial t}+\Delta \ell_+ +|\nabla
	\ell_+|^2-R-\frac{n}{2t}\leq 0,\label{eq_1_4}\\
	2\Delta \ell_+ +|\nabla \ell_+|^2-R-\frac{\ell_++n}{t}\leq 0,\label{eq_1_5}
	\end{align}
	where 
	\begin{align}\label{K}
	K=\int^t_0 \eta^{\frac{3}{2}}H(X)d\eta,
	\end{align}
	 $X$ is the velocity of the minimizing $\mathcal{L}_+$-geodesic connecting $(x,0)$ and $(y,t)$, and $$H(X)=\frac{\partial
		R}{\partial t}+2<\nabla R,X>+2Rc(X,X)+\frac{R}{t}$$ is Hamilton's trace Harnack. Furthermore, $\nabla \ell_+(y,t)=X(t)$ whenever the minimizing $\mathcal{L}_+$-geodesic connecting $(x,0)$ and $(y,t)$ is unique.
\end{thm}

\begin{thm}[Theorem 4.3 in \cite{LNi}]\label{Monotonicity}
	Let $(M,g(t))$ be a Ricci flow with bounded curvature at each time slice. Assume Hamilton's trace Harnack is nonnegative, then the quantity
	\begin{eqnarray*}
	\tilde{\theta}_+(t)=\frac{1}{(4\pi t)^{\frac{n}{2}}}\int_M e^{-\ell_+}d\mu_t
	\end{eqnarray*}
	is monotonically decreasing in $t$, where $\mu_t$ is the Riemannian measure of $g(t)$ and $\ell_+$ is the forward reduced distance based at some fixed point on $M\times\{0\}$.
\end{thm}
\begin{proof}[Sketch of proof]
Combing (\ref{eq_l_1}), (\ref{eq_l_2}), and (\ref{eq_l_3}), we have
\begin{eqnarray}\label{distribution}
\frac{\partial \ell_+}{\partial t}-\Delta \ell_+ +|\nabla
	\ell_+|^2+R+\frac{n}{2t}\geq \frac{K}{t^{\frac{3}{2}}}\geq 0
\end{eqnarray}
in the barrier sense or in the sense of distribution. Then, taking for granted the integration by parts at infinity, we compute
\begin{align*}
	&\frac{d}{d t}\int_{M} (4\pi t)^{-\frac{n}{2}}e^{-\ell_+(x,t)}d\mu_t\\
	=&-\int_{M}\left(
	\frac{\partial \ell_+}{\partial t}-\Delta \ell_+ +|\nabla
	\ell_+|^2+R+\frac{n}{2t}\right)(4\pi t)^{-\frac{n}{2}}e^{-\ell_+(x,t)}d\mu_t\\
	=&-\int_{M}\frac{K}{t^{\frac{3}{2}}} (4\pi t)^{-\frac{n}{2}}e^{-\ell_+(x,t)}d\mu_t\le 0.
	\end{align*}
	Note that the above computation is valid in our case according to the estimates in the next section.
	\end{proof}

\section{Estimates for $\ell_+$ function on Type III Ricci flows}

In this section, 
we use similar
methods as in \cite{N} to derive some estimates for the forward reduced
distance on a Type III immortal Ricci flow. Since Theorem \ref{TypeIII_Estimate}(2)---(4) are already well established in \cite{N}, we shall be relatively brief in their proofs. Recall that an immortal solution $(M,g(t))_{t\in[0,\infty)}$ is called Type III, if there exists a constant $C_0>0$, such that
\begin{eqnarray}\label{T_III}
|Rm|\leq\frac{C_0}{t} &\text{ everywhere on } &M\times(0,\infty).
\end{eqnarray}

\begin{thm}\label{TypeIII_Estimate}
	Let $(M^n,g(t))_{t\in[0,\infty)}$ be an $n$-dimensional Type III  immortal Ricci flow with nonnegative Ricci curvature everywhere. 	Let $ \ell_+$ be the forward reduced distance function based a fixed point $(x,0)$. Furthermore, assume that there exists a sequence of space-time points $\{(x_j,t_j)\}_{j=1}^\infty$ such that $t_j\nearrow\infty$ and
	\begin{equation}\label{key_assumption}
	\ell_+(x_j,t_j)\le A<\infty \ \text{ for all }\ j\geq 1.
	\end{equation}
Then, there exists a positive constant $Q$
	depending  only on $A$, $\alpha\in(0,1)$, $n$, and the Type III constant $C_0$ in (\ref{T_III}), such that the following inequalities hold for all
	$(y,t)\in M^n\times [\alpha^{-1},\alpha]$ and for all $j\geq 1$, understood in the barrier sense if any differentiation is involved.
	\begin{enumerate}
	\item  $\displaystyle |K^j|(y,t)\leq Q\sqrt{t}\left(1+\frac{d_{g_j(t)}(x_j,y)}{\sqrt{t}}\right)^2$, \ \ $\displaystyle |\nabla K^j|(y,t)\leq Q\left(1+\frac{d_{g_j(t)}(x_j,y)}{\sqrt{t}}\right)^2$, \\$\displaystyle \left|\frac{\partial}{\partial t} K^j\right|(y,t)\leq \frac{Q}{\sqrt{t}}\left(1+\frac{d_{g_j(t)}(x_j,y)}{\sqrt{t}}\right)^2$, 
		\item $\displaystyle\frac{1}{Q}\left(1+\frac{d_{g_j(t)}(x_j,y)}{\sqrt{t}}\right)^2-Q\leq \ell^j_+(y,t)\leq
		Q\left(1+\frac{d_{g_j(t)}(x_j,y)}{\sqrt{t}}\right)^2$,
		\item $\displaystyle|\nabla \ell^j_+|(y,t)\leq
		\frac{Q}{\sqrt{t}}\left(1+\frac{d_{g_j(t)}(x_j,y)}{\sqrt{t}}\right)$ ,
		\item $\displaystyle\left|\frac{\partial \ell^j_+}{\partial t}\right|(y,t)\leq
		\frac{Q}{t}\left(1+\frac{d_{g_j(t)}(x_j,y)}{\sqrt{t}}\right)^2$,
	\end{enumerate}
	where $g_j(t):=t_j^{-1}g(tt_j)$ is the Ricci flow sequence obtained by Type III scaling, $K^j$ is as defined in (\ref{K}) for the scaled Ricci flow $g_j$, and $\ell_+^j(\cdot,t):=\ell_+(\cdot,tt_j)$ is the forward reduced distance based at $(x,0)$ and with respect to the Ricci flow $g_j(t)$.
\end{thm}
\begin{proof}
	In this proof, we will use the capital letter $C$ to denote a general estimation constant, which depends on $\alpha$, $n$, $A$, and $C_0$ as indicated in the statement of this theorem, and could vary from line to line. Since the Ricci flow is Type III, by Shi's
	gradient estimates  
	we have 
	\begin{eqnarray*}|R|(y,t)\le \frac{C}{t}, \ \ \ |\nabla \cur|(y,t)\leq
	\frac{C}{t^{\frac{3}{2}}}, \ \ \ \left|\frac{\partial R}{\partial t}\right|(y,t)\leq \frac{C}{t^2},  \ \ \ \left|\nabla\frac{\partial}{\partial t}R\right|\leq\frac{C}{t^{\frac{5}{2}}},
		\end{eqnarray*}
for all $(y,t)\in M\times (0,\infty)$. Let $(y,t)\in M\times (0,\infty)$ be such that the minimizing $\mathcal{L}_+$-geodesic $\gamma$ connecting $(x,0)$ and $(y,t)$ is unique. We denote $X:=\gamma'$ and calculate that
	\begin{align}\label{ineq_K}
	|K|
	&=\left|\int^t_0 \eta^{\frac{3}{2}}\left(\frac{\partial
		R}{\partial \eta}+\frac{R}{\eta}+2\langle\nabla R,X\rangle+2Rc(X,X)\right)d\eta\right| \leq \int^t_0 \eta^{\frac{3}{2}}\left(\frac{C}{\eta^2}+\frac{C}{\eta^{\frac{3}{2}}}|X|+\frac{C}{\eta}|X|^2\right)d\eta\nonumber\\
	&\leq \int_0^t\frac{C}{\eta^{\frac{1}{2}}}d\eta+\int_0^t\sqrt{\eta}|X|^2d\eta\leq 2C\sqrt{t}+C\int_0^t\sqrt{\eta}(|X|^2+R)d\eta=2C\sqrt{t}(1+\ell_+(y,t)).
	\end{align}

	Since the Type III scaling process does not alter the Type III constant in (\ref{T_III}), we obtain from (\ref{eq_l_1}), (\ref{eq_l_2}),  and (\ref{ineq_K}) that
	\begin{align}\label{ineq_partial_t}
	\left|\frac{\partial{\ell^j_+}}{\partial t}\right|(y,t)&\leq \frac{C}{t}(1+\ell^j_+(y,t)),\\
	|\nabla \ell^j_+|^2(y,t)&\leq \frac{C}{t}(1+\ell^j_+(y,t)).\label{ineq_gradient}
	\end{align}
In view of the fact
	\begin{align*}
	\ell^j_+(x_j,1)=\ell_+(x_j,t_j)\le A,
	\end{align*} 
we may integrate (\ref{ineq_partial_t}) to obtain
	\begin{align}\label{ineq_l_up}
	\ell^j_+(x_j,t)\le C \ \ \text{ for all }   \ \ t\in [\alpha^{-1},\alpha].
	\end{align} 
	Integrating (\ref{ineq_gradient}) in space and applying (\ref{ineq_l_up}), we have
	\begin{align}\label{ineq_l_up_b}
	\ell^j_+(y,t) \leq Q\left(1+\frac{d_{g_j(t)}(x_j,y)}{\sqrt{t}}\right)^2 \ \ \text{ for all }\ \ (y,t)\in M\times[\alpha,\alpha^{-1}].
	\end{align}
	Conclusion (3), (4), the second inequality of conclusion (2), and the first inequality of conclusion (1) now follow from (\ref{ineq_K}), (\ref{ineq_partial_t}), (\ref{ineq_gradient}), and (\ref{ineq_l_up_b}).

	To obtain the first inequality of conclusion (2), we fix $(y,t)\in M\times[\alpha,\alpha^{-1}]$ and let $\gamma_1(s)$ and $\gamma_2(s)$ be minimizing $\mathcal{L}_+$-geodesics from $(x,0)$ to $(x^j,t)$ and to $(y,t)$, respectively, all with respect to $g_j(t)$. We denote $f(s):=d_{g_j(s)}(\gamma_1(s),\gamma_2(s))$. Then we have
	\begin{align*}
	\frac{d^-}{d
		s}f(s)
	&=\langle\nabla d_{g_j(s)},\gamma_1'(s)\rangle+\langle\nabla
	d_{g_j(s)},\gamma_2'(s)\rangle+\left(\frac{d^-}{d
		\tau}d_{g_j(\tau)}(\gamma_1(s),\gamma_2(s))\right)\Bigg|_{\tau=s}\\
	&=\langle\nabla d_{g_j(s)},\nabla \ell^j_+(\gamma_1(s),s)\rangle+\langle\nabla
	d_{g_j(s)},\nabla \ell^j_+(\gamma_2(s),s)\rangle +\left(\frac{d^-}{d
		\tau}d_{g_j(\tau)}(\gamma_1(s),\gamma_2(s))\right)\Bigg|_{\tau=s}\\
	&\leq |\nabla \ell^j_+(\gamma_1(s),s)|+|\nabla \ell^j_+(\gamma_2(s),s)|+	\left(\frac{d^-}{d
		\tau}\int_{\sigma}\sqrt{g_j(\tau)(\sigma',\sigma')}\right)\Bigg|_{\tau=s},
	\end{align*}
	where $\sigma$ is a unit speed minimizing geodesic from $\gamma_1(s)$
	to $\gamma_2(s)$ with respect to metric $g_j(s)$, 
	and
	$\frac{d^{-} f}{d s} = \liminf\limits _{h \rightarrow 0^{+}} \frac{f(s+h)-f(s)}{h}$ is the lower forward Dini derivative.
	By Lemma 18.1 in \cite{CCGGI}, we have 
	\begin{align*}
	\left(\frac{d^-}{d
		\tau}\int_{\sigma}\sqrt{g_j(\tau)(\sigma',\sigma')}\right)\Bigg|_{\tau=s}=-\min\limits_{\eta\in\mathcal{Z}\left(s\right)}\int_{\eta}Rc_{g_j(s)}(\eta',\eta')\leq 0,
	\end{align*}
	where $\mathcal{Z}\left(s\right)$ denotes the set of all unit speed minimizing geodesics
	from $\gamma_1(s)$
	to $\gamma_2(s)$ with respect to metric $g_j(s)$.
	Then, by
	(\ref{ineq_gradient}), we get
	\begin{align}\label{ineq_2.4_1}
	\frac{d^-}{ds}	f(s)\leq \sqrt{\frac{C_2}{s}(1+\ell^j_+(\gamma_1(s),s))}+\sqrt{\frac{C_2}{s}(1+\ell^j_+(\gamma_2(s),s))}.
	\end{align}
On the other hand, since for all $s\in(0, t]$, we have
	\begin{align*}
		\ell^j_+(\gamma_2(s),s)&=\frac{1}{2\sqrt{s}}\int^s_0\sqrt{\eta}(R+|X|^2) d\eta
	\leq \frac{1}{2\sqrt{s}}\int^t_0\sqrt{\eta}(R+|X|^2)
		d\eta=\frac{\sqrt{t}}{\sqrt{s}}
	\ell^j_+(y,t),
	\\
	\ell^j_+(\gamma_1(s),s)&\leq\frac{\sqrt{t}}{\sqrt{s}}\ell^j_+(x_j,t)\leq C\frac{\sqrt{t}}{\sqrt{s}},
	\end{align*}
where in the latter formula we have applied (\ref{ineq_l_up}). Then, (\ref{ineq_2.4_1}) becomes
	\begin{align*}
	\frac{d^-}{d
		s}	f(s)\leq
	C\frac{t^{\frac{1}{4}}}{s^{\frac{3}{4}}}\left(1+\sqrt{1+l^j_+(y,t)}\right)\ \ \text{ for all }\ \ s\in(0,t].
	\end{align*}
Integrating the above inequality from $0$ to $t$, we obtain the first inequality of conclusion (2).

	Finally, to obtain the last two inequalities of conclusion (1), we let  $\gamma$ be a minimizing $\mathcal{L}_+$-geodesic with respect to $g$, which connects $(x,0)$ and $(y,t)$, and $Y$ an $\mathcal{L}_+$-Jacobi field along $\gamma$, satisfying $[X,Y]=0$ and 
\begin{eqnarray}
\nabla_XY=\ric(Y)+\frac{1}{2\eta}Y,\ \ \ |Y(\eta)|^2=\frac{\eta}{t}|Y(t)|^2=\frac{\eta}{t},\ \ \ |\nabla_X Y|\leq \frac{C}{t^{\frac{1}{2}}\eta^{\frac{1}{2}}}.
\end{eqnarray}
Then, we may compute
\begin{align}\label{K_g}
|\delta_YK|&=\Bigg|\int_0^t\eta^{\frac{3}{2}}\bigg(\Big\langle\nabla\frac{\partial}{\partial\eta}R, Y\Big\rangle+\frac{1}{\eta}\langle\nabla R,Y\rangle+2\langle\nabla^2R,X\otimes Y\rangle+2\langle\nabla R,\nabla_XY\rangle\\\nonumber
&\quad +2\nabla\ric(Y,X,X)+2\ric(X,\nabla_XY)\bigg)d\eta\Bigg|\leq C\int_0^t\eta^{\frac{3}{2}}\left(\frac{1}{t^{\frac{1}{2}}\eta^2}+\frac{1}{t^{\frac{1}{2}}\eta^{\frac{3}{2}}}|X|+\frac{1}{t^{\frac{1}{2}}\eta}|X|^2\right)d\eta
\\\nonumber
&\leq \frac{C}{\sqrt{t}}\int_0^t\frac{1}{\eta^{\frac{1}{2}}}d\eta+\frac{C}{\sqrt{t}}\int_0^t\sqrt{\eta}(|X|^2+R)d\eta\leq C(1+\ell_+(y,t)).
\end{align}
Since the Type III constant in (\ref{T_III}) is not affected by the Type III scaling, we obtain the second inequality of conclusion (1) by (\ref{K_g}). Next, we observe that
\begin{align*}
\left|\frac{\partial}{\partial t}K\right|(y,t)&=\left|\frac{d}{d\eta}\Big|_{\eta=t}K-\langle\nabla K, X\rangle(y,t)\right|\leq t^{\frac{3}{2}}\left|\frac{\partial
		R}{\partial t}+\frac{R}{t}+2\langle\nabla R,X\rangle+2Rc(X,X)\right|+|\nabla K||X|
		\\\nonumber
		&\leq\frac{C}{\sqrt{t}}+C|X|+C\sqrt{t}|X|^2+C(1+\ell_+(y,t))|X|\leq \frac{C}{\sqrt{t}}(1+\ell_+(y,t))+C\sqrt{t}|\nabla \ell_+|^2(y,t)
		\\\nonumber
		&\leq\frac{C}{\sqrt{t}}(1+\ell_+(y,t)).
\end{align*}
Here we have used $X(t)=\nabla\ell_+(y,t)$ and formula (\ref{ineq_gradient}). Again, by the scaling invariance of the Type III constant, we obtain the third inequality of conclusion (1).

\end{proof}

\section{the proof of the main theorem}

In this section, we prove the main theorem by implementing the method of, \cite{CZhang}, \cite{LZ}, and \cite{Zh}. Since our techniques and arguments are very similar to that of \cite{CZhang}, we are not including obvious modifications, and the readers are referred to this paper for more detailed treatment.

Let $(M,g_0(t))$ be an expanding breather as described in the statement of Theorem \ref{main}. Note that the assumptions therein guarantees the validity of Hamilton's Harnack estimate (\cite{brendle}, \cite{cao} and \cite{RH2}). After rescaling and translating in time, we may assume that there exists $\alpha\in(0,1)$ and a diffeomorphism $\phi:M\to M$, such that
\begin{equation}\label{breather}
\alpha g_0(1)=\phi^*g_0(0).
\end{equation}

Similar to \cite{CZhang}, for each $j\geq 0$ we define 
\begin{eqnarray*}
\displaystyle t_j&=&\sum^j_{k=0}\alpha^{-k},\ \ t_0=1,
\\
g_j(t)&=&\alpha^{-j}(\phi^j)^*g_0(\alpha^{j}( t- t_{j-1})),\ \  t\in [ t_{j-1}, t_j].
\end{eqnarray*}
Obviously, $t_j\nearrow\infty$. We may then define the spliced immortal solution 
\begin{eqnarray}\label{defined_ancient_solution}
g( t)=\left\{
\begin{array}{ll}
g_0( t), \quad &  t\in [0,1], \\\nonumber
g_j( t), \quad & t\in [ t_{j-1}, t_j].
\end{array}
\right.
\end{eqnarray}
It is straightforward to check that 
$
g_j(t_{j-1})=g_{j-1}(t_{j-1})
$
and
$
|Rm_{g_{j}( t)}|
\le \frac{C}{ t}
$
for all $j\ge 1$, where the constant $C$ depends only on $\alpha$ and the curvature bound of the original breather.
 Then the immortal solution $g(t)$ is of Type III and is smooth by the uniqueness of the Ricci flow (c.f. \cite{uniqueness1} and \cite{uniqueness2}).

Next, we fix a point $p_0\in M$ and define $x_j=\phi^{-{(j+1)}}(p_0)$ for $j\ge 0$. Let $\ell_+$ be the reduced distance from $(p_0,0)$, and we shall proceed to show
\begin{eqnarray}\label{nonsense}
\limsup_{j\rightarrow\infty}\ell_+(x_{j},t_{j})<\infty.
\end{eqnarray} 
Let $\sigma:[0,1]\to M^n$ be a smooth curve satisfying $\sigma(0)=p_0$ and $\sigma(1)=x_0=\phi^{-1}(p_0)$.
We define $\sigma_j:[t_j,t_{j+1}]\rightarrow M$ and $\gamma_j:[0, t_{j+1}]\to M$ as
\begin{align}
&\sigma_j( t)=\phi^{-{(j+1)}}\circ\sigma(\alpha^{j+1}( t- t_{j})),\quad  t\in
[ t_{j}, t_{j+1}],
\\
&\gamma_j( t)=\left\{ \label{gamma}
\begin{array}{ll}
\sigma( t), \quad &  t\in [0,1], \\
\sigma_i( t), \quad & t\in [ t_{i}, t_{i+1}],0\le i\le j.
\end{array}
\right.
\end{align}
Since
\begin{equation*}
\sigma_j( t_{j})=\phi^{-(j+1)}\circ\sigma(0)=\phi^{-j}\circ\sigma(1)=\phi^{-j}\circ\sigma(\alpha^{j}( t_{j}- t_{j-1}))=\sigma_{j-1}( t_{j}),
\end{equation*}
we have that $\gamma_j(t)$ defined as (\ref{gamma}) is a piecewise smooth
$C^0$ curve satisfying $\gamma_j(0)=p_0$, $\gamma_j( t_j)=x_{j}$. 
We then compute
\begin{align*}
2\sqrt{ t_{j+1}}\ell_+(x_{j+1}, t_{j+1})&\le\mathcal{L}_+(\sigma)+\sum\limits_{i=1}^{j}
\int\limits_{ t_{i}}^{ t_{i+1}}\sqrt{ t}\left(R(\sigma_i( t), t)+|\sigma_i'( t)|^2_{g( t)}\right)d t\\
&\le D+\sum\limits_{i=1}^{j}
\int\limits_{ t_{i}}^{ t_{i+1}}\sqrt{ t}\left(\frac{C}{ t}+A\alpha^{i+1}\right)d t\le D+C\sum\limits_{i=1}^{j}
\alpha^{-\frac{i+1}{2}}
\\
&\le D+ C\alpha^{-\frac{j+1}{2}},
\end{align*}
for all $j\geq 0$, where  $A:=\max\limits_{ t\in [0,1]}|\sigma'( t)|_{g_0( t)}$, $C$, and $D$ are all constants independent of $j$. Since  $t_{j+1}\ge \alpha^{-(j+1)}$, we obtain (\ref{nonsense}) from the above computation. Theorem \ref{TypeIII_Estimate} is now applicable to $(M,g(t))_{t\in[0,\infty)}$ along the space-time sequence $\{(x_j,t_j)\}_{j=1}^\infty$.

From the construction of $g(t)$, we easily observe that
$$
\Big(M, t_j^{-1}g( t_j t),x_j\Big)_{t\in \left[1,\frac{ t_{j+1}}{ t_{j}}\right]}\rightarrow \Big(M,g_\infty(t),p_0\Big)_{t\in[1,\alpha^{-1}]},
$$
where $g_\infty$ and $g_0$ differ only by a scaling and a time-shifting. Furthermore, by Theorem \ref{TypeIII_Estimate}, we have that there exists a function $\ell_+^\infty: M\times[1,\alpha^{-1}]\rightarrow \mathbb{R}$, such that
\begin{eqnarray*}
\ell_+^j\rightarrow\ell_+^\infty
\end{eqnarray*}
in the $C^{0,\alpha}_{\operatorname{loc}}$ sense and the weak $*W^{1,2}_{\operatorname{loc}}$ sense, where $\ell_+^j(\cdot,t)=\ell_+(\cdot,tt_j)$. Our next goal is to show that $\ell_+^\infty$ gives rise to an expander structure on $(M,g_\infty(t))$.

Arguing as in section 4 of \cite{N} or section 6 of \cite{CZhang} and 
applying the monotonicity in theorem \ref{Monotonicity}  in the same way as one have applied  Perelman's reduced geometry, we have that $\ell_+^\infty: M\times[1,\alpha^{-1}]\rightarrow \mathbb{R}$ is a smooth function, satisfying
\begin{eqnarray}\label{linfty}
\frac{\partial \ell_+^\infty}{\partial t}-\Delta \ell_+^\infty +|\nabla
	\ell_+^\infty|^2+R_\infty+\frac{n}{2t}=0.
\end{eqnarray}

Furthermore, by Theorem \ref{TypeIII_Estimate}(1), we may find a function $K^\infty: M\times[1,\alpha^{-1}]\rightarrow\mathbb{R}$, such that $K^j\rightarrow K^\infty$ in the $C^{0,\alpha}_{\operatorname{loc}}$ sense and the weak $*W^{1,2}_{\operatorname{loc}}$ sense. Fixing arbitrary $1<s_1<s_2<\alpha^{-1}$ and an arbitrary nonnegative and smooth time-independent cut-off function $\varphi$ compactly supported on $M$, we obtain the following by (\ref{distribution})
\begin{align*}
0&\leq\int_{s_1}^{s_2}\int_M\frac{1}{t^{\frac{3}{2}}}K^j\varphi (4\pi t)^{-\frac{n}{2}}e^{-\ell_+^j}d\mu_t^jdt
\\
&\leq\int_{s_1}^{s_2}\int_M\left(\varphi\frac{\partial}{\partial t}\ell_+^j-\langle\nabla\varphi,\nabla\ell_+^j\rangle+R_j\varphi+\frac{n}{2t}\varphi\right) (4\pi t)^{-\frac{n}{2}}e^{-\ell_+^j}d\mu_t^jdt
\\
&\rightarrow\int_{s_1}^{s_2}\int_M\left(\varphi\frac{\partial}{\partial t}\ell_+^\infty-\langle\nabla\varphi,\nabla\ell_+^\infty\rangle+R_\infty\varphi+\frac{n}{2t}\varphi\right) (4\pi t)^{-\frac{n}{2}}e^{-\ell_+^\infty}d\mu_t^\infty dt
\\
&=0.
\end{align*}
This shows that
\begin{eqnarray}
K^\infty\equiv 0 &\text{ everywhere on }& M\times[1,\alpha^{-1}].
\end{eqnarray}
Since equations (\ref{eq_l_1}) and (\ref{eq_l_2}) are both carried to the limit in the sense of distribution, and since $\ell_+^\infty$ is smooth, the following hold on $M\times[1,\alpha^{-1}]$
\begin{eqnarray*}
\frac{\partial}{\partial t}\ell_+^\infty=R_\infty-\frac{\ell_+^\infty}{t},\ \ \ |\nabla \ell_+^\infty|^2=\frac{\ell_+^\infty}{t}-R_\infty.
\end{eqnarray*}
In combination with (\ref{linfty}), we then have
\begin{align*}
\frac{\partial \ell_+^{\infty}}{\partial t}+\Delta \ell_+^{\infty} +|\nabla
\ell_+^{\infty}|^2-R_{{\infty}}-\frac{n}{2t}=0,
\\	
2\Delta \ell_+^{\infty} +|\nabla \ell_+^{\infty}|^2-R_{{\infty}}-\frac{\ell_+^{\infty}+n}{t}=0.
\end{align*}
Then, by Theorem 1.2 in \cite{FIN},
\begin{align*}
&\left(\frac{\partial}{\partial t}+\Delta-R_{\infty}\right) \left(t(2\Delta \ell_+^{\infty} +|\nabla \ell_+^{\infty}|^2-R_{\infty})-\ell_+^{\infty}-n\right)(4\pi t)^{-\frac{n}{2}}e^{\ell_+^{\infty}}\\
=&-2t \left|Rc_{\infty}-\nabla \nabla \ell_+^{\infty}+\frac{g_{\infty} }{2t}\right|^{2}(4\pi t)^{-\frac{n}{2}}e^{\ell_+^{\infty}}=0.
\end{align*}
Hence $\ell_+^{\infty}$ satisfies the following gradient expanding soliton equation
\begin{align*}
Rc_{{\infty}}-\nabla \nabla \ell_+^{\infty}= -\frac{1}{2t }g_{\infty}.
\end{align*}
Note that Theorem \ref{TypeIII_Estimate}(2) guarantees that $\displaystyle (4\pi t)^{-\frac{n}{2}}e^{\ell_+^{\infty}}>0$ everywhere. This finishes the proof.

\end{document}